\documentclass[reqno,9pt]{amsart}

\usepackage{epsf}
\usepackage{graphics}
\usepackage{graphicx, color, mathtools, euscript}
\usepackage{amssymb}
\usepackage{amsmath}
\usepackage{mathcomp}
\usepackage{wasysym}
\usepackage{enumitem}

\date{}

\theoremstyle{plain}
\newtheorem{theorem}{Theorem}

\newtheorem*{acknowledgment}{Acknowledgment}

\theoremstyle{definition}

\theoremstyle{remark}

\newtheorem{remark}{Remark}

\def\N{{\mathbb N}}

\newcommand{\Aeul}{\EuScript{A}}
\newcommand{\Beul}{\EuScript{B}}

\title{An asymptotic version of the union-closed sets conjecture}

\author{Luca Studer}

\begin{document}

\begin{abstract} We show that the biggest possible average set size in the complement $2^{\{1,2,\ldots, n\}} \setminus  \Aeul$ of a union-closed family $\Aeul \subset 2^{\{1,2, \ldots, n\}}$ is $\tfrac{n+1}{2}$. 
With the same proof we get a sharp upper bound for the average frequency in complements of union-closed families. This implies an asymptotic version of the union-closed sets conjecture, formulated in terms of complements of union-closed families. 
\end{abstract}

\maketitle

Let $n \in \N$, $[n] =\{1,2, \ldots, n\}$ and let $2^{[n]}=\{A: A \subset [n]\}$ be the power set on $n$ elements. A family $\Aeul \subset 2^{[n]}$ is called \textit{union-closed} if $A,B \in \Aeul$ implies $A \cup B \in \Aeul$. 
The union-closed sets conjecture asserts that if $\Aeul \subset 2^{[n]}$ is union-closed, then 
there is $k \in [n]$ such that $|\{A \in \Aeul: k \in A\}|/|\Aeul| \geq \frac{1}{2}$; or formulated in terms of the complement $\Beul \coloneqq 2^{[n]} \setminus \Aeul$ of a union-closed familiy $\Aeul \subset 2^{[n]}$, the conjecture states that there is 
$k \in [n]$ such that $|\{B \in \Beul: k \in B\}|/|\Beul| \leq \frac{1}{2}$ (for a survey article on the conjecture see~\cite{Bruhn and Schaudt}). We show that asymptotically the latter formulation is true, even when the minimum of $|\{B \in \Beul: k \in B\}|$ over $k \in [n]$ is replaced by the average
$$\mu(\Beul)\coloneqq \frac{1}{n} \sum_{k=1}^n |\{B \in \Beul: k \in B\}|.$$
\begin{theorem}
\label{1}
If $\Beul =2^{[n]}\setminus \Aeul$ is the complement of a union-closed family $\Aeul\subset 2^{[n]}$, then 
\begin{enumerate}
\item[(i)] $\sum_{B \in \Beul} |B| \leq \frac{n+1}{2} |\Beul|,$
\item[(ii)] $\mu(\Beul) \leq \frac{n+1}{2n}|\Beul|.$
\end{enumerate}
In particular, if $n_l$, $l \in \N$ is a positive integer sequence and $\Aeul_l \subset 2^{[n_l]}$ is a sequence of union-closed families with $\Aeul_l\not =\Aeul_{l'}$ for $l\not = l'$, then the complements $\Beul_l=2^{[n_l]}\setminus \Aeul_l$ satisfy $$\limsup_{l \to \infty} \frac{\mu(\Beul_l)}{|\Beul_l|} \leq \frac{1}{2}.$$
\end{theorem}

\begin{remark}
All inequalities in Theorem~\ref{1} are sharp as can be seen by considering the union-closed family $\Aeul=\{A \subset [n]: 1 \not \in A\}$ with complement $\Beul=\{B \subset [n]: 1 \in B\}$.
\end{remark}

Theorem~\ref{1} contrasts the fact that a similar weakening of the union-closed sets conjecture stated in terms of union-closed families (instead of their complements) seems very hard. Concretely, there are union-closed families $\Aeul \subset 2^{[n]}$ with $\mu(\Aeul)<\frac{1}{100}|\Aeul|$, and it is unknown if for every union-closed family $\Aeul \subset 2^{[n]}$ there is $k \in [n]$ with $|\{A \in \Aeul: k \in A\}| \geq \frac{1}{100}|\Aeul|$. The following remark is crucial for the given proof of Theorem~\ref{1}.

\begin{remark}
\label{property}
If $\Beul= 2^{[n]}\setminus \Aeul$ is the complement of a union-closed family $\Aeul \subset 2^{[n]}$, $B \in \Beul$ and $k,l \in B$ are distinct, then $B\setminus \{k\} \in \Beul$ or $B\setminus \{l\} \in \Beul$. Indeed, if $B \setminus \{k\}, B \setminus \{l\} \in \Aeul$, then 
the union $B=B \setminus \{k\} \cup B \setminus \{l\}$ is also in $\Aeul$ (and thus not in $\Beul$).
\end{remark}

\begin{remark}
Similarly to the recent work of Karpas~\cite{Karpas}, who showed that the union-closed sets conjecture holds for union-closed families $\Aeul\subset 2^{[n]}$ with $|\Aeul|\geq 2^{n-1}$, the given proof of Theorem~\ref{1} depends only on the property formulated in Remark~\ref{property}. 
\end{remark}

\begin{proof}[Proof of Theorem~\ref{1}] Define 
\begin{align*}
U & \coloneqq \{(B,k):B \in \Beul, k \in B, B \setminus  \{k\} \in \Beul\}, \\
V & \coloneqq \{(B,k):B \in \Beul, k \in B, B \setminus \{k\} \not \in \Beul\}, \\ 
W & \coloneqq \{(B,k):B \in \Beul, k \not \in B, B \cup \{k\} \in \Beul\}, \\
X & \coloneqq \{(B,k):B \in \Beul, k \not \in B, B \cup \{k\} \not \in \Beul\}.
\end{align*}
Note that $U,V,W,X$ are pairwise disjoint and 
$$U\cup V \cup W \cup X=\Beul \times [n].$$ We get $|U|+|V|+|W|+|X|=n|\Beul|$.
Moreover, $(B,k) \mapsto (B \cup \{k\}, k)$ defines a bijection $W \to U$. This gives $|W|=|U|$. Together we get
$$|U|+|V| = \frac{|U|+|W|}{2}+|V|=\frac{n|\Beul|-|V|-|X|}{2} +|V| = \frac{n|\Beul|+|V|-|X|}{2}\leq  \frac{n|\Beul|+|V|}{2}.$$
It follows directly from Remark~\ref{property} that $|V|\leq |\Beul|$, hence together with the last inequality
$$|U|+|V|\leq \frac{n+1}{2}|\Beul|.$$ Assertion (i) follows now from 
$$\sum_{B \in \Beul} |B|=|\{(B,k): B \in \Beul, k \in B\}|=|U \cup V|=|U|+|V|,$$ and similarly, assertion (ii) follows from 
$$n \mu(\Beul)=\sum_{k \in [n]} |\{B \in \Beul: k \in \Beul\}|=|\{(B,k): B \in \Beul, k \in B\}|=|U \cup V|=|U|+|V|.$$
To get the asymptotic result, note that for fixed $n \in \N$ there are at most finitely many distinct union-closed families on the ground set $[n]$ ($2^{2^n}$ is a trivial upper bound). Therefore, since $\Aeul_l \subset 2^{[n_l]}$, $l \in \N$ is a sequence of union-closed families without repetition, we have $n_l \to \infty$ as $l \to \infty$.
Together with (ii) we get 
$$\limsup_{l \to \infty} \frac{\mu(\Beul_l)}{|\Beul_l|}  \leq \limsup_{n \to \infty} \Big(\frac{1}{2}+\frac{1}{2n_l} \Big) =\frac{1}{2},$$
as desired.
\end{proof}

\begin{remark}
Alternatively, Theorem~\ref{1} can be proved building on Reimer's work about the average set size in union-closed families \cite{Reimer}. However, the above proof seemed more natural.
\end{remark}

\begin{acknowledgment}
\normalfont
I would like to thank Ilan Karpas and Sebastian Baader for valuable comments.
\end{acknowledgment}

\bigskip
\noindent
Facultad de Ciencias Matem\' aticas de la Universidad Complutense, Plaza Ciencias, 28040, Madrid, Spain

\smallskip
\noindent
\texttt{luca.studer@gmail.com}

\end{document}